\documentclass{article}[10pt]
\usepackage{theorem}
\usepackage{eurosym}
\usepackage{amssymb}
\usepackage{float}
\usepackage{amsmath}
\usepackage{amsfonts}
\usepackage[american]{babel}
\usepackage{verbatim}
\usepackage{geometry}

\newtheorem{theorem}{Theorem}

{\theorembodyfont{\rmfamily}

}

\newtheorem{lemma}[theorem]{Lemma}

{\theorembodyfont{\rmfamily}
\newtheorem{remark}[theorem]{Remark}
}

\newenvironment{proof}[1][Proof]{\noindent\textbf{#1} }{\ \rule{0.5em}{0.5em}}

\newcommand*\re{\mathbb{R}}

\newcommand*\Omegabar{\overline{\Omega}}
\newcommand*\delomega{\partial\Omega}

\newcommand*\WBDisk{W^{1,2}_0(B_1)\cap \mathcal{B}_1(B_1)}
\newcommand*\WBDiskrad{W^{1,2}_{0,rad}(B_1)\cap \mathcal{B}_1(B_1)}

\begin{document}

\title{The Singular Moser-Trudinger Inequality on Simply Connected Domains}
\maketitle
\date{}

\centerline{\scshape Gyula Csat\'{o}$^1$ and Prosenjit Roy$^2$ }
\medskip
{\footnotesize
\centerline{1 Departamento de Matem\'atica, Universidad de Concepci\'on, Concepcion, Chile.}
  \centerline{2 Tata Institute of Fundamental Research, Centre For Applicable Mathematics, Bangalore, India.}
   \centerline{gy.csato.ch@gmail.com, prosenjit@math.tifrbng.res.in}

}

\smallskip

\begin{abstract} In this paper the authors complete their study of the singular Moser-Trudinger embedding [G. Csat\'o  and P. Roy, Extremal functions for the singular Moser-Trudinger inequality in 2 dimensions, Calc. Var. Partial Differential Equations, DOI 10.1007/s00526-015-0867-5], abbreviated [CR]. There they have proven the existence of an extremal function for the singular Moser-Trudinger embedding
$$
  \sup_{\substack {v\in W^{1,2}_0(\Omega) \\ \|\nabla v\|_{L^2}\leq 1}}
  \int_{\Omega}\frac{e^{\alpha v^2}-1}{|x|^{\beta}}\leq C,
$$
where $\alpha>0$ and $\beta\in [0,2)$ are such that $\frac{\alpha}{4\pi}+\frac{\beta}{2}\leq1,$ and $\Omega\subset\re^2.$ This generalizes a well known result by Flucher, who has proven the case $\beta=0.$ The proof in [CR] is however far too technical and complicated for simply connected domains. Here we give a much simpler and more self-contained proof using complex analysis, which also generalizes the corresponding proof given by Flucher for such domains. This should make [CR] more easily accessible.
\end{abstract}

\let\thefootnote\relax
\let\thefootnote\relax

\section{Introduction}\label{section:introduction}

The Moser-Trudinger embedding has been generalized by Adimurthi-Sandeep \cite{Adi-Sandeep} to a singular version, which reads as the following:
If $\alpha>0$ and $\beta\in[0,2)$ is such that
\begin{equation}
 \label{intro:eq:alpha and beta sum}
  \frac{\alpha}{4\pi}+\frac{\beta}{2}\leq 1,
\end{equation}
then the following supremum is finite
\begin{equation}\label{intro:eq:Adi Sandeep result}
  \sup_{\substack {v\in W^{1,2}_0(\Omega) \\ \|\nabla v\|_{L^2}\leq 1}}
  \int_{\Omega}\frac{e^{\alpha v^2}-1}{|x|^{\beta}}<\infty\,,
\end{equation}
where $\Omega\subset\re^2$ is a bounded open smooth set.
In Csat\'o-Roy \cite{Csato-Roy} the authors have proven that the supremum is attained. This generalizes the result of Flucher \cite{Flucher}, who has proven the case $\beta=0.$ On the history of Flucher's result and other recent developments on the subject we refer to \cite{Adi-Tintarev}, \cite{Carleson-Chang},   \cite{Flucher},   \cite{Malchiodi-Martinazzi} , \cite{manchini} and \cite{Struwe 1}. Flucher gives two different proofs, one for simply connected domains and one for general domains. The same can be done for the case $\beta>0,$ however even more technical and substantial difficulties arise, leading to lengthy proofs. Therefore we have decided to split these two cases into separate papers. In the present paper we shall give a significantly simpler proof of the following theorem.

\begin{theorem}\label{theorem:intro:Extremal for Singular Moser-Trudinger}
Let $\Omega\subset\re^2$ be a bounded open simply connected set with $0\in\Omega$ and smooth boundary $\delomega.$ Let $\alpha>0$ and $\beta\in[0,2)$ be such that \eqref{intro:eq:alpha and beta sum} is satisfied.
Then there exists $u\in W_0^{1,2}\left(\Omega\right)$ such that $\|\nabla u\|_{L^2(\Omega)}\leq 1$ and
$$
  \sup_{\substack {v\in W^{1,2}_0(\Omega) \\ \|\nabla v\|_{L^2}\leq 1}}
  \int_{\Omega}\frac{e^{\alpha v^2}-1}{|x|^{\beta}}=\int_{\Omega}\frac{e^{\alpha u^2}-1}{|x|^{\beta}}.
$$
\end{theorem}
Let us explain the crucial simplifications of the proof for simply connected domains. $F_{\Omega}$ shall denote the singular Moser-Trudinger functional, $F_{\Omega}^{sup}$ its supremum given by \eqref{intro:eq:Adi Sandeep result} and $F_{\Omega}^{\delta}(0)$ the maximizing concentration level at $0,$ cf. Section \ref{section:notation}. 
The proof is based on the result of Carleson-Chang \cite{Carleson-Chang}, which easily extends to the singular Moser-Trudinger functional and states that on the unit ball $B_1$ we have that
$$
  F_{B_1}^{\delta}(0)<F_{B_1}^{sup}.
$$
This implies by a concentration compactness alternative that on the ball the supremum is attained.
Therefore, as in Flucher \cite{Flucher}, the main difficulty in \cite{Csato-Roy} consists in relating $F^{sup}_{\Omega}$ to $F^{sup}_{B_1}$, respectively $F^{\delta}_{\Omega}(0)$ to $F^{\delta}_{B_1}(0).$ This consists of two parts. 
\smallskip

\textit{Part 1.} One establishes the inequality 
\begin{equation}\label{intro:ball to domain}
  F_{\Omega}^{\sup}\geq I_{\Omega}(0)^{2-\beta}F_{B_1}^{\sup},
\end{equation} 
where $I_{\Omega}(0)$ is the conformal incenter of $\Omega$ at $0.$
The conformal incenter $I_{\Omega}(x)$ of a domain is defined by the Green's function for the Laplace operator $G_{\Omega,x}$ of $\Omega$ with singularity at $x,$ and its regular part $H_{\Omega,x},$ namely
$$
  G_{\Omega,x}(y)=-\frac{1}{2\pi}\log(|x-y|)-H_{\Omega,x}(y),\quad\text{ and }\quad I_{\Omega}(x)=e^{-2\pi H_{\Omega,x}(x)}.
$$
For simply connected domains there exists a conformal map $h$ with the properties
$$
  h:B_1\to \Omega\quad\text{ and }\quad h(0)=0.
$$
It is unique up to composition by rotation $h(e^{i\varphi}z)$ for $\varphi\in\re.$ Since $G_{B_1,0}=G_{\Omega,0}\circ h$ one easily obtains, see Flucher \cite{Flucher}, that $I_{\Omega}(0)$ can everywhere be replaced by
$$
I_{\Omega}(0)=|h'(0)|.
$$
In particular, in this paper, no knowledge about the Green's function, conformal incenter and its properties is required.
The proof of \eqref{intro:ball to domain} consists of constructing for any given radial function $v$ on the ball a corresponding function $u$ given on $\Omega,$ which satisfies the inequality $F_{\Omega}(u)\geq I_{\Omega}(0)^{2-\beta}F_{B_1}(v).$ This is done by defining $u$ as
\begin{equation}\label{intro:ball to domain by Green}
  u(y)=v\left(e^{-2\pi G_{\Omega,0}(y)}\right)=v\left((G_{B_1,0})^{-1}( G_{\Omega,0}(y))\right).
\end{equation}
The proof of the inequality \eqref{intro:ball to domain} follows then from a careful analysis of the transormation \eqref{intro:ball to domain by Green} using the coarea formula, some fine properties of the Green's function and, most importantly,   
a singularly weighted isoperimetric inequality. This isoperimetric inequality is of independent interest with many other consequences and has been established in a separate paper in Csat\'o \cite{Csato}. For simply connected domains the transformation \eqref{intro:ball to domain by Green} can be written as
\begin{equation}
 \label{intro:ball to domain transfor in simply conn.}
  u=v\circ h^{-1}.
\end{equation}
Note that \eqref{intro:ball to domain transfor in simply conn.} also makes sense if $v$ is not radial.
With this a direct proof is given avoiding the above mentioned difficulties and which is moreover independent of Csat\'o \cite{Csato}.

\smallskip 

\textit{Part 2.} Using a transformation for concentrating sequences $\{u_i\}\subset W_0^{1,2}(\Omega)$ one proves a kind of reverse inequality to \eqref{intro:ball to domain}, namely
\begin{equation}\label{intro:domain to ball inequality}
  F^{\delta}_{\Omega}(0)\leq I_{\Omega}^{2-\beta}(0)F_{B_1}^{\delta}(0).
\end{equation}
On simply connected domains this construction is simple, because the transformation \eqref{intro:ball to domain transfor in simply conn.} is invertible and one defines
$$
  v_i=u_i\circ h
$$
to obtain the proof of \eqref{intro:domain to ball inequality}. For general domains there is no simple construction, because the transformation \eqref{intro:ball to domain by Green} is not invertible. The proof is therefore long and technical using among others the following ingredients: existence and regularity for the Laplace equation, certain compact embedding results for H\"older spaces, approximation of Sobolev functions by smooth ones, Sard's theorem, a capacity argument for $W_0^{1,2}$ functions, Bocher's theorem, Schwarz symmetrization and a careful analysis of the properties of the Green's function near its singularity. In the proof for simply connected domains some well known but powerful theorems from complex analyis are sufficient and none of the previously mentioned tools is required.

\section{Notations and Preliminaries}\label{section:notation}

Throughout this paper $\Omega\subset\re^2$ will denote a bounded simply connected open set with $0\in\Omega$ and smooth boundary $\delomega.$ Balls with radius $R$ and center at $x$ are written $B_R(x)\subset\re^2;$ if $x=0,$ we simply write $B_R$. The space $W^{1,2}(\Omega)$ denotes the usual Sobolev space of functions and $W^{1,2}_0(\Omega)$ those Sobolev functions with vanishing trace on the boundary. Throughout this paper $\alpha,\beta\in \re$ are two constants satisfying $\alpha>0,$ $\beta\in[0,2)$ and 
$$
  \frac{\alpha}{4\pi}+\frac{\beta}{2}\leq1.
$$
We define the functional $F_{\Omega}:W_0^{1,2}(\Omega)\to \re$ by
\begin{equation}\label{definition of F Omega}
 F_{\Omega}(u)=\int_{\Omega}\frac{e^{\alpha u^2}-1}{|x|^{\beta}}\,dx.
\end{equation}
We say that a sequence $\{u_i\}\subset W_0^{1,2}(\Omega)$ concentrates at $x\in\Omegabar$ if 
$$
  \lim_{i\to\infty}\|\nabla u_i\|_{L^2}=1\quad\text{ and }\quad\forall\;\epsilon>0\quad
  \lim_{i\to\infty}\int_{\Omega\backslash B_{\epsilon}(x)}|\nabla u_i|^2=0.
$$
We will use the following well known property of concentrating sequences: if $\{u_i\}$ concentrates, then $u_i\rightharpoonup 0$ in $W^{1,2}(\Omega),$ i.e. converges weakly to zero. In particular
\begin{equation} 
 \label{eq:properties of concentrating sequences}
  u_i\to 0\quad\text{ in }L^2(\Omega),
\end{equation}
see for instance Flucher \cite{Flucher} Step 1 page 478. We define the sets
\begin{align*}
  W^{1,2}_{0,rad}(B_1)&=\left\{u\in W^{1,2}_0(B_1)\,\big|\, u\text{ is radial }\right\}
  \smallskip \\
    \mathcal{B}_1(\Omega)&=\left\{u\in W^{1,2}_0(\Omega)\,\big|\,\|\nabla u\|_{L^2}\leq 1\right\}.
\end{align*}
By abuse of notation we will usually write $u(x)=u(|x|)$ for $u\in W^{1,2}_{0,rad}(B_1).$ We define
$$
F_{\Omega}^{\text{sup}}=\sup_{u\in \mathcal{B}_1(\Omega)}F_{\Omega}(u).
$$
If $x\in\Omegabar$ and the supremum is taken only over concentrating sequences, we write $F_{\Omega}^{\delta}(x),$ more precisely
$$
  F_{\Omega}^{\delta}(x)=\sup\left\{\limsup_{i\to\infty}F_{\Omega}(u_i)\,\Big|\quad \{u_i\}\subset \mathcal{B}_1(\Omega)\text{ concentrates at } x\right\}.
$$
We now repeat those preliminary results which we use from \cite{Csato-Roy}, respectively which have essentially been established by other authors in previous works. The next two Lemmas are both applications of the Vitali convergence theorem, see \cite{Csato-Roy} for a detailed proof.

\begin{lemma}\label{lemma:compactness in the interior}
Let $0\leq\eta<1$ and suppose $\{u_i\}\subset W_0^{1,2}(\Omega)$ is such that
$$
  \limsup_{i\to\infty}\|\nabla u_i\|_{L^2}\leq \eta\quad\text{and}\quad 
  u_i\rightharpoonup u\text{ in }W^{1,2}(\Omega)
$$
for some $u\in W^{1,2}_0(\Omega).$ Then for some subsequence
$$
  \frac{e^{\alpha u_i^2}}{|x|^{\beta}}\to \frac{e^{\alpha u^2}}{|x|^{\beta}}\quad\text{ in }L^1(\Omega)
$$
and in particular $\lim_{i\to\infty}F_{\Omega}(u_i)=F_{\Omega}(u).$
\end{lemma}

\begin{remark}
Theorem \ref{theorem:intro:Extremal for Singular Moser-Trudinger}, for the case when $\frac{\alpha}{4\pi} +\frac{\beta}{2} < 1$,  is an easy consequence of the above lemma, cf. \cite{Csato-Roy}.
\end{remark}

\begin{lemma}
\label{proposition:if u_i concentrates somewhere else than zero}
Let $\beta>0,$ $\{u_i\}\subset  \mathcal{B}_1(\Omega)$ and suppose that $u_i$ concentrates at $x_0\in\Omegabar,$ where $x_0\neq 0.$ Then 
one has that, for some subsequence, $u_i\rightharpoonup 0$ in $W^{1,2}(\Omega)$  and
$$
  \lim_{i\to\infty}F_{\Omega}(u_i)=F_{\Omega}(0)=0.
$$
In particular $F_{\Omega}^{\delta}(x_0)=0.$
\end{lemma}

The next theorem is essentially due to Lions \cite{Lions}. 

\begin{theorem}[Concentration-Compactness Alternative]\label{theorem:concentration alternative for singular moser trudinger}
Let $\{u_i\}\subset \mathcal{B}_1(\Omega).$ Then there is a subsequence and $u\in W_0^{1,2}(\Omega)$ with $u_i\rightharpoonup u$ in $W^{1,2}(\Omega),$  such that either

(a) $\{u_i\}$ concentrates at a point $x\in\Omegabar,$
\newline or

(b) the following convergence holds true
$$
  \lim_{i\to\infty}F_{\Omega}(u_i)=F_{\Omega}(u).
$$
\end{theorem}

\begin{remark}
 \label{remark:independence of topology}
Lemmas \ref{lemma:compactness in the interior}, \ref{proposition:if u_i concentrates somewhere else than zero} and Theorem \ref{theorem:concentration alternative for singular moser trudinger} do not require that $\Omega$ is simply connected and the difficulty of their proof is independent of the toplogy of the domain.
\end{remark}

The next theorem is the combination of the results of Carleson-Chang \cite{Carleson-Chang} and Adimurthi-Sandeep \cite{Adi-Sandeep}, see \cite{Csato-Roy} for a detailed proof.

\begin{theorem}
\label{theorem:supremum of FOmega on Ball} The following strict inequality holds: $F_{B_1}^{\delta}(0)<F_{B_1}^{\sup}.$
\end{theorem}

\begin{remark}
Theorem \ref{theorem:supremum of FOmega on Ball} together with Theorem \ref{theorem:concentration alternative for singular moser trudinger} implies that the supremum $F_{\Omega}^{sup}$ is attained if $\Omega=B_1$.
\end{remark}

\section{Proof of the Main Theorem via Riemman map}\label{Harmonic Transplantation}

By abuse of notation we will identify subsets $U\subset \re^2$ with subsets of the complex plain $U\subset \mathbb{C}.$ The set of holomorphic functions on $U$ will be denoted by $H(U).$
Throughout this section $h\in H(B_1)$ shall denote the conformal map, which exists by the Riemann mapping theorem, and which satisfies
$$
  h:B_1\to\Omega\quad\text{ and }\quad h(0)=0.
$$
The next theorem is  the analogue of the  ``ball to domain construction'', i.e. Theorem 16 in \cite{Csato-Roy}.

\begin{theorem}
\label{theorem:ball to general domain:sup inequality}
For any  $v\in W^{1,2}_{0,rad}(B_1)\cap \mathcal{B}_1(B_1)$ define $u=v\circ h^{-1}.$ Then $u\in   \mathcal{B}_1(\Omega)$ and it satisfies
$$
  F_{\Omega}(u)\geq |h'(0)|^{2-\beta}F_{B_1}(v).
$$
In particular the following inequality holds true
$$
  F_{\Omega}^{\text{sup}}\geq |h'(0)|^{2-\beta}F^{\sup}_{B_1}.
$$
\end{theorem}

For the proof of Theorem \ref{theorem:ball to general domain:sup inequality} we need the following lemma.

\begin{lemma}
\label{lemma:some complex analysis}
For any $\gamma,\beta\in \re$ the following inequality holds true
$$
  2\pi|h'(0)|^{\gamma-\beta}\leq r^{\beta}\int_0^{2\pi} \frac{\big|h'\big(re^{it}\big)\big|^{\gamma}}{\big|h\big(re^{it}\big)\big|^{\beta}}dt
  \quad\text{ for all }r\in(0,1).
$$
\end{lemma}

\begin{remark}
For this lemma it is actually sufficient that $0\in\Omega\subset\re^2$ is a simply connected open set such that $\Omega\neq \re^2.$
\end{remark}

\begin{proof} 
Since $h(0)=0,$ there exists a holomorphic map $g\in H(B_1)$ such that
$$
  h(z)=z g(z)\quad\text{ and }\quad g(0)=h'(0)\neq 0.
$$
Moreover, since $h$ is bijective, we must have that $h(z)\neq 0$ for all $z\in B_1\backslash\{0\}.$ This implies that
$$
  g\neq 0\quad\text{in }B_1.
$$
Since $h$ is conformal, we also have that $h'\neq 0$ in $B_1.$
Therefore there exists $\varphi,\psi\in H(B_1)$ (cf. for instance \cite{Rudin} Theorem 13.11) such that
$$
  g=\exp(\varphi)\quad\text{ and }\quad h'=\exp(\psi)\quad\text{ in }B_1,
$$
where $\exp$ is the exponential map. We therefore obtain that
$$
  \frac{\exp(\gamma \psi)}{\exp(\beta\varphi)}\in H(B_1).
$$
Note that for any $\eta\in\re$ and any $z\in\mathbb{C}$ we have that $|\exp(\eta z)|=|\exp(z)|^{\eta}.$ Using the Cauchy integral mean value formula, we get
\begin{align*}
  |h'(0)|^{\gamma-\beta}
  =&
 \frac{|\exp(\psi(0))|^{\gamma}}{|\exp(\varphi(0))|^{\beta}}
 =\left|\frac{\exp(\gamma \psi(0))}{\exp(\beta\varphi(0))}\right|
 =
 \frac{1}{2\pi}\left|\int_0^{2\pi} \frac{\exp\big(\gamma \psi\big(re^{it}\big)\big)}
 {\exp\big(\beta\varphi\big(re^{it}\big)\big)}dt\right| 
 \smallskip \\
 \leq&\frac{1}{2\pi}\int_0^{2\pi} \frac{\big|\exp\big( \psi\big(re^{it}\big)\big)\big|^{\gamma}}
 {\big|\exp\big(\varphi\big(re^{it}\big)\big)\big|^{\beta}}dt
 =
 \frac{r^{\beta}}{2\pi}\int_0^{2\pi} \frac{\big|h'\big(re^{it}\big)\big|^{\gamma}}{\big|h\big(re^{it}\big)\big|^{\beta}}dt.
\end{align*}
This proves the lemma.
\end{proof}
\smallskip

\begin{proof}[Proof of Theorem \ref{theorem:ball to general domain:sup inequality}.] \textit{Step 1.}
It follows from the Cauchy-Riemann equations that if we consider $h$ as a diffeomorphism between the two open sets $\Omega,B_1\subset\re^2$, then the Jacobian calculates as
$$
  \det Dh(y)=|h'(y)|^2.
$$
Using again the Cauchy-Riemann equations we also obtain that
$$
  |\nabla v(y)|^2=\nabla u(h(y))  Dh(y) Dh(y)^t (\nabla u(h(y)))^t= |\nabla u(h(y))|^2|h'(y)|^2,
$$
where $A^t$ is the transpose of a matrix $A.$ It thus follows by change of variables that
$$
  \int_{\Omega}|\nabla u|^2=\int_{B_1}|\nabla v|^2.
$$
This shows that $u\in \mathcal{B}_1(\Omega)$ if $v\in \mathcal{B}_1(B_1)$ and therefore $F_{\Omega}(u)$ is well defined. Using again the change of variables $x=h(y),$ we get
$$
  F_{\Omega}(u)=\int_{h(B_1)}\frac{e^{\alpha u^2}-1}{|x|^{\beta}}=
  \int_{B_1}\frac{e^{\alpha v(y)^2}-1}{|h(y)|^{\beta}}|h'(y)|^2 dy.
$$
Using that $v$ is radial gives
$$
  F_{\Omega}(u)=\int_0^1\big(e^{\alpha v(r)^2}-1\big)\left(\int_{\partial B_r} \frac{|h'(y)|^2}{|h(y)|^{\beta}}d\sigma\right)dr.
$$
From Lemma \ref{lemma:some complex analysis}, and using again that $v$ is radial, we get
$$
  F_{\Omega}(u)\geq |h'(0)|^{2-\beta}\int_0^1\frac{e^{\alpha v(r)^2}-1}{r^{\beta}} 2\pi r=|h'(0)|^{2-\beta}F_{B_1}(v).
$$
This proves  the first statement of the theorem.
\smallskip

\textit{Step 2.} Let us prove the second statement. Let $v\in \WBDisk$ and let $v^{\ast}$ be its radially decreasing symmetric rearrangement. From the properties of symmetric rearrangements (see for instance Kesavan \cite{Kesavan}) we have that $v^{\ast}\in \WBDiskrad$ and
$$
  F_{B_1}(v)\leq F_{B_1}(v^{\ast}).
$$
Let $u=v^{\ast}\circ h^{-1}\in W^{1,2}_0(\Omega).$ Then by Step 1, we get $u\in \mathcal{B}_1(B_1)$ and
$$
  F_{\Omega}^{sup}\geq F_{\Omega}(u)\geq |h'(0)|^{2-\beta}F_{B_1}(v^{\ast})\geq |h'(0)|^{2-\beta}F_{B_1}(v).
$$
Since $v$ was arbitrary, the second statement is proven.
\end{proof}

\smallskip

The next theorem is the analogue of the ``domain to ball construction'', i.e. Theorem 21 and Propostion 22 in \cite{Csato-Roy}.

\begin{theorem}
\label{theorem:concentration formula by domain to ball}
Let $\{u_i\}\subset  \mathcal{B}_1(\Omega)$  be a sequence which concentrates at $0.$ Define $v_i$ by
$$
  v_i=u_i\circ h\in \mathcal{B}_1(B_1).
$$
Then $\{v_i\}$ concentrates at $0$ and
$$
  \lim_{i\to\infty}F_{\Omega}(u_i)=|h'(0)|^{2-\beta} 
  \lim_{i\to\infty}F_{B_1}(v_i),
$$
if either of the limits exist.
In particular the following identity holds
$$
  F_{\Omega}^{\delta}(0)=|h'(0)|^{2-\beta}F_{B_1}^{\delta}(0).
$$
\end{theorem}

\begin{proof} 
\textit{Step 1.}
As in the proof of Theorem \ref{theorem:ball to general domain:sup inequality}, we can show by a change of variables, that indeed $v_i\in \mathcal{B}_1(B_1),$ and thus $F_{B_1}(v_i)$ is well defined.
To calculate $\lim_{i\to\infty}F_{B_1}(v_i)$ we  use again the same change of variables $x=h(y),$  and obtain that
$$
  \lim_{i\to\infty} F_{\Omega}(u_i)=\lim_{i\to\infty}\int_{h(B_1)} \frac{e^{\alpha u_i^2}-1}{|x|^{\beta}}=\lim_{i\to\infty}\int_{B_1}\frac{e^{\alpha v_i^2}-1}{|h(y)|^{\beta}}|h'(y)|^2.
$$
Let $\delta>0$ be arbitrary and let us split the integral in two parts
\begin{align*}
  \lim_{i\to\infty} F_{\Omega}(u_i)=&
  \lim_{i\to\infty}\int_{B_{\delta}}\frac{e^{\alpha v_i^2}-1}{|h(y)|^{\beta}}|h'(y)|^2
  +
  \lim_{i\to\infty}\int_{B_1\backslash B_{\delta}}\frac{e^{\alpha v_i^2}-1}{|h(y)|^{\beta}}|h'(y)|^2
  \smallskip \\
  =&\lim_{i\to\infty}A_1^i(\delta)+\lim_{i\to\infty}A_2^i(\delta).
\end{align*}
\smallskip 

\textit{Step 2.} In this step we show that
\begin{equation}\label{eq:proof:limit of A two i delta limit}
 \lim_{i\to\infty}A_2^i(\delta)=
  \lim_{i\to\infty}\int_{B_1\backslash B_{\delta}}\frac{e^{\alpha v_i(y)^2}-1}{|h(y)|^{\beta}}|h'(y)|^2dy=0\quad\text{ for all }\delta>0.
\end{equation}
Since $h(z)\neq 0$ for all $z\in B_1\backslash\{0\},$ we obtain that 
$$
  \frac{|h'(y)|^2}{|h(y)|^{\beta}}\in L^{\infty}\left(B_1\backslash B_{\delta}\right).
$$
Thereby we have also used that $|h'|^2$ is bounded up to the boundary $\delomega.$ This follows from the fact that $|h'|^2=\det Dh$ and $h\in C^1(\overline{B_1}),$ because $\Omega$ is bounded and has smooth boundary (cf. for instance Theorem 5.2.4 page 121 in Krantz \cite{Krantz Complex Analysis})
Therefore it is enough to prove that
$$
  \lim_{i\to\infty}\int_{B_1\backslash B_{\delta}}\left(e^{\alpha v_i^2}-1\right)=0\quad\text{ for all }\delta>0.
$$
Choose $\eta\in C^{\infty}\left(\overline{B_1}\right)$ such that $\eta\geq 0$ and
$$
  \eta=1\quad\text{ in }B_1\backslash B_{\delta},\qquad\eta=0\quad\text{ in }B_{\delta/2}.
$$
Then we obtain that
\begin{equation}\label{eq:proof:eta vi and limit in Bdeltao}
  \lim_{i\to\infty}\int_{B_1\backslash  B_{\delta}}\left(e^{\alpha v_i^2}-1\right)
  \leq
  \limsup_{i\to\infty}\int_{B_1\backslash B_{\delta/2}}\left(e^{\alpha (\eta v_i)^2}-1\right).
\end{equation}
Note that $\eta v_i\in W^{1,2}_0\left(B_1\backslash\overline{B}_{\delta/2}\right)$ and the gradient can be estimated as
\begin{align*}
  \int_{B_1\backslash B_{\delta/2}}|\nabla(\eta v_i)|^2\leq &
  2\int_{B_1\backslash B_{\delta/2}}|v_i\,\nabla\eta |^2
  +
  2\int_{B_1\backslash B_{\delta/2}}\eta^2|\nabla v_i|^2
  \smallskip \\
  \leq &
  C({\eta},\delta)\int_{B_1}|v_i|^2+2\int_{B_1\backslash B_{\delta/2}}|\nabla v_i|^2,
\end{align*}
for some constant $C(\eta,\delta)\in\re.$
It can be easily verified (similarly as in Step 1 in the proof of Theorem \ref{theorem:ball to general domain:sup inequality}) that $v_i$ concentrates at $0,$ since $h(0)=0.$
Therefore both terms on the right hand side tend to $0$ for $i\to\infty,$ see \eqref{eq:properties of concentrating sequences}. In particular we get that, for some $i_0\in\mathbb{N},$
$$
  \int_{B_1\backslash B_{\delta/2}}|\nabla(\eta v_i)|^2\leq \frac{1}{2} \quad\text{ for all }i\geq i_0\,.
$$
We can therefore apply Lemma \ref{lemma:compactness in the interior} (see Remark \ref{remark:independence of topology}) for the sequence $\eta v_i$ and the domain $B_1\backslash B_{\delta/2}.$ This gives, using \eqref{eq:proof:eta vi and limit in Bdeltao} and that $v_i\rightharpoonup 0$ in $W^{1,2}(B_1),$ that
$$
  \lim_{i\to\infty}\int_{B_1\backslash  B_{\delta}}\left(e^{\alpha v_i^2}-1\right)
  =0.
$$
which concludes the proof of \eqref{eq:proof:limit of A two i delta limit}.
\smallskip

\textit{Step 3.} Since $v_i$ concentrates at $0,$ we can show exactly as in Step 2, that
$$
  \lim_{i\to\infty}\int_{B_1\backslash B_{\delta}(0)}\frac{e^{\alpha v_i^2}-1}{|y|^{\beta}}=0.
$$
In particular
\begin{equation}\label{eq:proof:proposition:F B_1 of vi and delta}
  \lim_{i\to\infty}\int_{B_{\delta}(0)}\frac{e^{\alpha v_i^2}-1}{|y|^{\beta}}
  =
  \lim_{i\to\infty}\int_{B_1}\frac{e^{\alpha v_i^2}-1}{|y|^{\beta}}=
  \lim_{i\to\infty}F_{B_1}(v_i).
\end{equation}

\textit{Step 4.} Let $g\in H(B_1)$ be as in the proof of Lemma \ref{lemma:some complex analysis}. In particular $g(z)\neq 0$ for all $z\in B_1$ and
$$
  \chi(y)=\frac{|y|^{\beta}|h'(y)|^2}{|h(y)|^{\beta}}=\frac{|h'(y)|^2}{|g(y)|^{\beta}}
$$
defines a coninuous function on $B_1.$ Therefore, if $\epsilon>0$ is given, we can chose $\delta>0$ such that
$$
  |\chi(y)-\chi(0)|\leq \epsilon\quad\text{ for all }y\in B_{\delta}(0).
$$
Since $g(0)=h'(0)$  (see proof of Lemma \ref{lemma:some complex analysis}), we get
\begin{equation}\label{eq:proof:lemma chi at zero}
  \chi(0)=|h'(0)|^{2-\beta}.
\end{equation}
Finally, note that by definition of $\chi$
\begin{equation}\label{eq:proof:Ai d delta limit written with chi}
  \lim_{i\to\infty}A^i_1(\delta)=\lim_{i\to\infty}\int_{B_{\delta}}\frac{e^{\alpha v_i^2}-1}{|y|^{\beta}} \chi(y) dy.
\end{equation}

\textit{Step 5 (conclusion).} Let $\epsilon>0$ be given and choose $\delta$ as in Step 4. Then from Step 1, equations \eqref{eq:proof:limit of A two i delta limit} and \eqref{eq:proof:lemma chi at zero} we get that
\begin{align*}
 \left|\lim_{i\to\infty}F_{\Omega}(u_i)-|h'(0)|^{2-\beta}\lim_{i\to\infty}F_{B_1}(v_i)\right|=\left|\lim_{i\to\infty}A_1^i(\delta)-\chi(0)\lim_{i\to\infty}F_{B_1}(v_i)\right|.
\end{align*}
Finally we obtain from \eqref{eq:proof:proposition:F B_1 of vi and delta}, \eqref{eq:proof:Ai d delta limit written with chi} and from the choice of $\delta$ in Step 4, that
\begin{align*}
  \left|\lim_{i\to\infty}F_{\Omega}(u_i)-|h'(0)|^{2-\beta}\lim_{i\to\infty}F_{{B_1}}(v_i)\right|
  =&
  \left|\lim_{i\to\infty}\int_{B_{\delta}}\frac{e^{\alpha v_i^2}-1}{|y|^{\beta}}\left(\chi(y)-\chi(0)\right)\right|
  \smallskip \\
  \leq&
  \epsilon\, F^{sup}_{B_1}\,,
\end{align*}
where $F^{\sup}_{B_1}<\infty$ is the constant given by the singular Moser-Trudinger embedding, see \eqref{intro:eq:Adi Sandeep result}.
Since $\epsilon$ was arbitrary, this proves the theorem.
\end{proof}
\smallskip

We are now able to prove the main theorem.
\smallskip

\begin{proof}[Proof of Theorem \ref{theorem:intro:Extremal for Singular Moser-Trudinger}.]
From Theorems \ref{theorem:concentration formula by domain to ball}, \ref{theorem:supremum of FOmega on Ball} and \ref{theorem:ball to general domain:sup inequality} we know that
$$
  F_{\Omega}^{\delta}(0)=|h'(0)|^{2-\beta}F_{B_1}^{\delta}(0)< |h'(0)|^{2-\beta}F_{B_1}^{\sup}\leq F_{\Omega}^{\sup}.
$$
Thus we obtain, using also Lemma \ref{proposition:if u_i concentrates somewhere else than zero}, that $F_{\Omega}^{\delta}(x)<F_{\Omega}^{\sup}$ for all $x\in\Omegabar,$ if $\beta>0.$ If $\beta=0,$ the same holds true by the result of Flucher \cite{Flucher} (the proof is the same: one can do all the steps with a different $h:B_1\to\Omega,$ satisfying $h(x)=0.$ This leads to $F^{\delta}_{\Omega}(x)=|h'(x)|^2F^{\delta}_{B_1}(0)<F_{\Omega}^{sup}$). This implies that maximizing sequences cannot concentrate and the result follows from Theorem \ref{theorem:concentration alternative for singular moser trudinger}.
\end{proof}

\bigskip

\noindent\textbf{Acknowledgements} The research work of the second author is supported by "Innovation in Science Pursuit for Inspired Research (INSPIRE)" under the IVR Number: 20140000099.

\end{document}